\documentclass[11pt]{amsart}

\title[Universal admissibility for scattering transforms]{Universal admissibility for scattering transforms}

\usepackage[a4paper,top=2.5cm,bottom=2.5cm,hmargin=2.5cm,heightrounded]{geometry}
\usepackage[T1]{fontenc}
\usepackage[utf8]{inputenc}
\usepackage{lmodern}
\usepackage{microtype}
\usepackage{amsmath,amssymb,amsthm,mathtools}
\mathtoolsset{showonlyrefs}
\usepackage{cite}
\usepackage[hidelinks]{hyperref}
\usepackage{tikz}
\usetikzlibrary{arrows.meta,calc,positioning,decorations.pathreplacing}
\theoremstyle{plain}
\newtheorem{theorem}{Theorem}[section]

\newtheorem{lemma}[theorem]{Lemma}
\newtheorem{corollary}[theorem]{Corollary}
\theoremstyle{definition}
\newtheorem{definition}[theorem]{Definition}

\theoremstyle{remark}
\newtheorem{remark}[theorem]{Remark}

\newcommand{\R}{\mathbb R}

\newcommand{\N}{\mathbb N}
\newcommand{\Nzero}{\mathbb N_0}
\newcommand{\Z}{\mathbb Z}
\newcommand{\T}{\mathbb T}
\newcommand{\G}{G}
\newcommand{\Gdual}{\widehat G}

\newcommand{\eps}{\varepsilon}
\newcommand{\iu}{\mathrm{i}}
\newcommand{\wh}{\widehat}
\newcommand{\Id}{\mathrm{Id}}

\newcommand{\dist}{\operatorname{dist}}
\newcommand{\supp}{\operatorname{supp}}

\newcommand{\abs}[1]{\left\lvert #1 \right\rvert}
\newcommand{\norm}[1]{\left\lVert #1 \right\rVert}
\newcommand{\ip}[2]{\left\langle #1,#2\right\rangle}

\newcommand{\dd}{\,\mathrm{d}}
\newcommand{\dx}{\,\mathrm{d}x}
\newcommand{\dxi}{\,\mathrm{d}\xi}

\newcommand{\given}{\nonscript\,\delimsize\vert\nonscript\,}
\DeclarePairedDelimiterX{\set}[1]{\lbrace}{\rbrace}{#1}

\newcommand{\calJ}{\mathcal J}
\newcommand{\calO}{\mathcal O}
\newcommand{\calS}{\mathcal S}

\newcommand{\PathDepth}[1]{\Lambda^{#1}}
\newcommand{\PathsAll}{\Lambda^\ast}
\newcommand{\PathsLe}[1]{\Lambda^{\leq #1}}

\newcommand{\Scat}{\mathsf{S}_{\chi}}
\newcommand{\ScatEq}[1]{\mathsf{S}_{\chi, #1}}
\newcommand{\ScatLe}[1]{\mathsf{S}_{\chi,\leq #1}}
\newcommand{\ScatSet}[1]{\mathsf{S}_{\chi}[#1]}

\usepackage{orcidlink}

\author[Max Getter]{%
Max Getter
\orcidlink{0009-0001-1269-4623}%
}

\address{%
Chair for Geometry and Analysis,
RWTH Aachen University,
D-52062 Aachen, Germany
}
\email[Max Getter]{getter@mathga.rwth-aachen.de}


\begin{document}

\begin{abstract}
We resolve the admissibility problem for scattering transforms: every Parseval filter bank on $\mathbb{R}^d$ whose low-pass filter is nonvanishing at the origin induces a norm-preserving scattering transform, without requiring any analyticity, frequency-localization, or geometric covering assumptions. Furthermore, for any Sobolev input $f\in H^s(\mathbb{R}^d)$, we establish a universal decay bound of $\mathcal{O}(N^{-\min\{s,1\}/d})$ on the depth-$N$ residual norm. Finally, we construct a filter bank of band-limited Schwartz functions and a band-limited Schwartz input for which the low-pass multiplier equals one near the origin but the propagated energy decays subexponentially. This demonstrates that no universal exponential rate can hold under these assumptions alone.
\end{abstract}
\maketitle

\section{Introduction}

The scattering transform, introduced by Mallat in \cite{Mallat2012}, is a nonlinear multiscale representation obtained by alternating convolutions with prescribed filters and pointwise modulus nonlinearities, followed by low-pass averaging. Its architecture may be viewed as a mathematically tractable model of a deep convolutional neural network: the convolutional filters are fixed rather than learned, the modulus acts as the nonlinearity, and the outputs collected at successive depths form a hierarchical family of features. 
This rigid structure makes it possible to establish properties that remain difficult to prove for general neural networks, including nonexpansiveness, approximate translation invariance, and stability under geometric deformations; see, among others, \cite{Mallat2012,BrunaMallat2013,WiatowskiBolcskei2018,NicolaTrapasso2023}. 

The theory of scattering transforms also has notable connections to other areas. Invertibility questions are closely related to phase retrieval \cite{mallat2015phase,alaifari2019stable,fuhr2025cheeger}, while stability arguments interact with the stability theory of max filtering \cite{cahill2025group} and with paradifferential techniques from nonlinear PDE \cite{getter2026stability}. Scattering networks have also found applications ranging from image and audio processing to quantum chemistry and physics \cite{BrunaMallat2013,SprechmannBrunaLeCun2015,EickenbergExarchakisHirnMallatThiry2018,ChengMorelAllysMenardMallat2024,LicciardiCarboneRondoniNagar2025}. Motivated by these developments and by the intrinsically non-Euclidean structure of many natural data domains, scattering-type networks have recently been generalized beyond Euclidean settings; see \cite{bronstein2017geometric,ChewEtAl2024,getter2025digital,arias2025scattering} and the references therein.

A fundamental question is whether the features extracted over all network depths retain the full energy of the input. At each layer, the Parseval identity for the underlying filters decomposes the propagated energy into an observed part, extracted by the low-pass filter, and a residual part, passed to the next layer. Iterating this identity produces a nonincreasing sequence of residual energies. If this residual energy vanishes as the depth tends to infinity, the filter bank is said to be \emph{admissible}, and the scattering transform preserves the $L^2$-norm of the input \cite[Theorem~2.6]{Mallat2012}.

Admissibility has consequences beyond deterministic norm preservation. In Mallat's original theory \cite{Mallat2012}, energy decay also underpins an energy conservation principle for stationary processes and was used to establish limiting translation invariance. Although limiting translation invariance can now be recovered without admissibility \cite{CzajaKolstoeKoralov2024}, control of the energy propagated to deep layers remains important in the stability analysis of scattering representations \cite[Section~6.2]{getter2026stability}.

Historically, establishing admissibility has required restrictive assumptions on the underlying filters. Mallat proved energy preservation for wavelet systems satisfying a technical condition on the Fourier transform of the mother wavelet, which he termed wavelet admissibility \cite[Theorem~2.6]{Mallat2012}. While Mallat explicitly identified analytic cubic spline Battle--Lemari\'e wavelets in dimension one as admissible, explicit verifications for other mother wavelets appear to be scarce. In particular, to the best of our knowledge, no mother wavelet $\psi\in L^1(\R^d)\cap L^2(\R^d)$ in dimension $d\geq2$ has been identified for which this condition has been rigorously verified. Subsequent work developed alternative sufficient conditions and quantitative energy decay estimates for various structured systems; see, for example, \cite{Waldspurger2017,wiatowski2017energy,CzajaLi2019}. More recent contributions provide comprehensive treatments of this question for classical and generalized scattering networks in Euclidean \cite{FuhrGetter2025,kolstoe2026generalized} and non-Euclidean \cite{getter2025digital} settings. Despite these advances, the following basic question has remained unresolved:

\emph{Is every Parseval filter bank whose low-pass multiplier is nonzero at the origin admissible?}

The main contribution of this paper is to resolve this question in the affirmative. We establish that universal admissibility holds for \emph{any} scattering filter bank. More precisely, every Parseval filter bank whose low-pass filter is nonvanishing at the origin is admissible: the propagated energy converges to zero for every square-integrable input, yielding a norm-preserving feature map. No analyticity, frequency localization, or geometric covering assumptions are required. We complement this qualitative statement with a universal polynomial decay estimate for Sobolev inputs. Finally, we show that these assumptions alone cannot imply exponential decay, even when all filters and the input are band-limited Schwartz functions and the low-pass multiplier is identically one near the origin.

The generality of the main result allows it to apply to scattering constructions based on Parseval filter banks of wavelet, Gabor, curvelet, and shearlet type, whenever the integrability and low-pass assumptions in Definition~\ref{def:parseval-bank} are satisfied. These families exhibit markedly different frequency geometries: wavelets resolve frequency according to scale, Gabor systems provide essentially uniform time-frequency resolution, and curvelet- and shearlet-type constructions use increasingly anisotropic directional channels. We refer to \cite{Mallat_book2009,Grochenig2001,CandesDonoho2004,GuoLabate2007} for the corresponding frame constructions. Our arguments use none of these particular geometries and depend only on the Parseval identity and the presence of a genuine low-pass channel.

\subsection{Scattering architecture and filter banks}
We now introduce the precise framework used throughout the paper, beginning with the Fourier normalization and the underlying filter-bank class. For $h\in L^1(\R^d)$, we set
\begin{equation*}
    \wh h(\xi)\coloneqq\int_{\R^d}h(x)e^{-2\pi\iu\xi\cdot x}\dx,\qquad
    \check h(x)\coloneqq\int_{\R^d}h(\xi)e^{2\pi\iu\xi\cdot x}\dxi.
\end{equation*}
With these conventions, the Fourier transform extends to a unitary automorphism of $L^2(\R^d)$. 
The class of filter banks studied in this paper is the following.
\begin{definition}\label{def:parseval-bank}
    A \emph{scattering filter bank} on $\R^d$ is a pair $(\chi,\Psi)$, where $\Psi=(\psi_\lambda)_{\lambda\in\Lambda}$ is indexed by a countable set $\Lambda$, and $\chi,\psi_\lambda\in L^1(\R^d)\cap L^2(\R^d)$ for every $\lambda\in\Lambda$, such that
    \begin{equation}\label{eq:LP-condition}
        \abs{\wh\chi(\xi)}^2+\sum_{\lambda\in\Lambda}\abs{\wh\psi_\lambda(\xi)}^2=1
    \end{equation}
    for almost every $\xi\in\R^d$, and such that $\wh\chi(0)\neq0$.
\end{definition}

The assumption that the filters belong to $L^1(\R^d)$ ensures that their Fourier transforms are continuous. In particular, $\wh\chi(0)\neq0$ implies that there exist a radius $R_\chi>0$ and a constant $c_\chi>0$ such that
\begin{equation}\label{eq:low-pass_pointwise_lower_Fourier_bound}
    \abs{\wh\chi(\xi)}\geq c_\chi \qquad \text{for every }\xi\in B_{R_\chi}(0).
\end{equation}
By contrast, the Parseval identity \eqref{eq:LP-condition} is asserted only almost everywhere. Indeed, continuity of the individual Fourier transforms does not in general imply continuity of the infinite sum appearing there. 

We briefly relate Definition~\ref{def:parseval-bank} to the standard frame formulation.
\begin{remark}
    Let $(g_i)_{i\in I}\subset L^1(\R^d)\cap L^2(\R^d)$ be a semi-discrete frame, so that
    \[A\norm{f}_2^2\leq\sum_{i\in I}\norm{f*g_i}_2^2\leq B\norm{f}_2^2\]
    for every $f\in L^2(\R^d)$ and some $0<A\leq B<\infty$. By Plancherel's theorem, this is equivalent to
    \[A\leq m(\xi)\coloneqq\sum_{i\in I}\abs{\wh g_i(\xi)}^2\leq B\]
    for almost every $\xi\in\R^d$. The associated Parseval normalization is formally given by
    \begin{equation}\label{eq:parseval-normalization}
        \wh{g_i^\circ}(\xi)\coloneqq m(\xi)^{-1/2}\wh g_i(\xi),\qquad i\in I,
    \end{equation}
    for almost every $\xi\in\R^d$. Since $m$ is almost everywhere bounded above and bounded away from zero, each $g_i^\circ$ belongs to $L^2(\R^d)$ and
    \[\sum_{i\in I}\abs{\wh{g_i^\circ}(\xi)}^2=1\]
    for almost every $\xi\in\R^d$. Whenever the normalized filters additionally belong to $L^1(\R^d)$ and some $i_0\in I$ satisfies $\wh{g_{i_0}^\circ}(0)\neq0$, setting $\chi=g_{i_0}^\circ$ and $\Psi=(g_i^\circ)_{i\in I\setminus\{i_0\}}$ produces a scattering filter bank. In the special case of a tight frame, where $A=B$, the normalization reduces to $g_i^\circ=A^{-1/2}g_i$ and therefore preserves the $L^1$-regularity of the filters.
\end{remark}

Fix a scattering filter bank $(\chi,\Psi)$ on $\R^d$, with $\Psi=(\psi_\lambda)_{\lambda\in\Lambda}$. A path of length $n\geq 1$ is a tuple
\[p=(\lambda_1,\ldots,\lambda_n)\in\PathDepth{n}.\]
We set $\PathDepth{0}\coloneqq\{e\}$, where $e$ denotes the empty path, and write
\[\PathsAll\coloneqq\bigcup_{n\geq0}\PathDepth{n}.\]
For $f\in L^2(\R^d)$, let $U[e]f\coloneqq f$. For $\lambda\in\Lambda$, define
\[U[\lambda]f\coloneqq\abs{f*\psi_\lambda},\]
and, for $p=(\lambda_1,\ldots,\lambda_n)\in\PathDepth{n}$ with $n\geq1$, set
\[U[p]f\coloneqq U[\lambda_n]\cdots U[\lambda_1]f.\]
Young's convolution inequality shows that $U[p]f\in L^2(\R^d)$ for every path $p\in\PathsAll$. The scattering coefficient associated with $p$ is
\[\Scat[p]f\coloneqq (U[p]f)*\chi.\]
For a path set $P\subset\PathsAll$, we use the notation
\begin{equation*}
    U[P]f\coloneqq\bigl(U[p]f\bigr)_{p\in P},\qquad \ScatSet{P}f\coloneqq\bigl(\Scat[p]f\bigr)_{p\in P},\qquad U_nf\coloneqq U[\PathDepth{n}]f.
\end{equation*}
For $N\in\Nzero$, let
\[\PathsLe{N}\coloneqq\bigcup_{0\leq n\leq N}\PathDepth{n}.\]
The full, depth-$N$-truncated, and $N$th-layer scattering transforms are defined by
\begin{equation*}
    \Scat f\coloneqq\ScatSet{\PathsAll}f,\qquad \ScatLe{N}f\coloneqq\ScatSet{\PathsLe{N}}f,\qquad \ScatEq{N}f\coloneqq\ScatSet{\PathDepth{N}}f.
\end{equation*}
When no ambiguity can arise, we identify an index $\lambda\in\Lambda$ with its associated filter $\psi_\lambda$ in path notation; in particular, we also write $U[\psi_\lambda]$ for $U[\lambda]$ and use the analogous notation for iterated propagators.

To state the fundamental layerwise energy decomposition, define the energy extracted at depth $N$ by
\[A_N(f)\coloneqq\norm{\ScatEq{N}f}_{\ell^2(\PathDepth{N};L^2(\R^d))}^2=\sum_{p\in\PathDepth{N}}\norm{\Scat[p]f}_2^2\]
and the residual energy at depth $N$ by
\[W_N(f)\coloneqq\norm{U_Nf}_{\ell^2(\PathDepth{N};L^2(\R^d))}^2=\sum_{p\in\PathDepth{N}}\norm{U[p]f}_2^2.\]
The Parseval identity \eqref{eq:LP-condition}, Plancherel's theorem, and the identity $\norm{\abs{g}}_2=\norm{g}_2$ then yield
\begin{equation}\label{eq:layerwise-energy-identity}
    A_N(f)=W_N(f)-W_{N+1}(f).
\end{equation}
In particular, $(W_n(f))_{n\in\Nzero}$ is nonincreasing and nonnegative, so its limit $W_\infty(f)\geq 0$ exists. For later reference, let us note that this entails the asymptotic vanishing 
\begin{equation}\label{eq:lim-A_N}
    \lim_{N\to\infty} A_N(f)=0.
\end{equation}
The layerwise energy decomposition also allows us to quantify the cumulative energy of the scattering coefficients up to order $N$ in terms of the residual energy by summing \eqref{eq:layerwise-energy-identity}, which yields 
\begin{equation}\label{eq:finite-energy-telescoping}
    \norm{\ScatLe{N}f}_{\ell^2(\PathsLe{N};L^2(\R^d))}^2=\sum_{n=0}^N\sum_{p\in\PathDepth{n}}\norm{\Scat[p]f}_2^2=\sum_{n=0}^NA_n(f)=\norm{f}_2^2-W_{N+1}(f).
\end{equation}
Thus the full scattering transform preserves the norm precisely when the energy propagated to arbitrarily deep layers vanishes.

\begin{definition}\label{def:admissibility}
    A scattering filter bank $(\chi,\Psi)$ is called \emph{admissible} if
    \[\lim_{N\to\infty}\norm{U_Nf}_{\ell^2(\PathDepth{N};L^2(\R^d))}=0\]
    for every $f\in L^2(\R^d)$.
\end{definition}
This terminology refers to the vanishing-tail property itself and should be distinguished from the sufficient wavelet condition termed admissibility in \cite[Theorem~2.6]{Mallat2012}.

\subsection{Main results}

Our first result shows that every scattering filter bank in the sense of Definition~\ref{def:parseval-bank} is admissible.

\begin{theorem}\label{thm:intro-universal-admissibility}
    Let $(\chi,\Psi)$ be a scattering filter bank on $\R^d$. Then $(\chi,\Psi)$ is admissible. More precisely, for every $f\in L^2(\R^d)$,
    \[\lim_{N\to\infty}\norm{U_Nf}_{\ell^2(\PathDepth{N};L^2(\R^d))}=0.\]
    Consequently,
    \begin{equation*}
        \norm{\Scat f}_{\ell^2(\PathsAll;L^2(\R^d))}^2=\sum_{n=0}^\infty\sum_{p\in\PathDepth{n}}\norm{\Scat[p]f}_2^2=\norm{f}_2^2,
    \end{equation*}
    so $\Scat$ is norm-preserving from $L^2(\R^d)$ into $\ell^2(\PathsAll;L^2(\R^d))$.
\end{theorem}

The theorem eliminates the need for any auxiliary frequency-side assumptions, whether on the input signal or on the underlying filters, that appeared in previous energy-conservation results. The only local information required about the filter bank is the lower bound \eqref{eq:low-pass_pointwise_lower_Fourier_bound} on the output-generating low-pass filter; the remaining filters may have arbitrary and mutually unrelated frequency supports, subject only to the Parseval identity.

The central observation is that every propagated signal $U[p]f$ associated with a nonempty path is nonnegative. Its autocorrelation $\ip{U[p]f}{T_xU[p]f}$ is therefore nonnegative, so it cannot be reduced by oscillatory cancellation. More precisely, one constructs a nonnegative function $w\in L^1(\R^d)\cap C_0(\R^d)$, bounded from below near the origin, such that
\[\int_{\R^d}w(x)\ip{h}{T_xh}\dx\lesssim_\chi\norm{h*\chi}_2^2\]
for every nonnegative $h\in L^2(\R^d)$. Since convolution and modulus are translation covariant and nonexpansive in $L^2$, persistence of a fixed amount of residual energy would imply a uniform lower bound for the corresponding weighted correlations. The low-pass energy extracted at those depths would therefore remain bounded away from zero, in contradiction with \eqref{eq:lim-A_N}.

A quantitative version of this argument yields an explicit decay rate when the input has Sobolev regularity.

\begin{corollary}\label{cor:intro-Hs-polynomial-decay}
    Let $(\chi,\Psi)$ be a scattering filter bank on $\R^d$, and let $s>0$. Then there exists a constant $C>0$, depending only on $s$, $d$, and the low-pass filter $\chi$, such that
    \begin{equation*}
        \norm{U_Nf}_{\ell^2(\PathDepth{N};L^2(\R^d))}\leq C \norm{f}_{H^s(\R^d)} N^{-\frac{\min\{s,1\}}{d}}
    \end{equation*}
    for every $f\in H^s(\R^d)$ and every $N\in\N$.
\end{corollary}

Notably, Corollary~\ref{cor:intro-Hs-polynomial-decay} is uniform with respect to the high-pass family $\Psi$. Once the Parseval condition is imposed, the decay constant depends only on the dimension, the input regularity, and the low-pass filter. The exponent is obtained by combining the correlation argument with the standard translation estimate
\begin{equation*}
    \norm{f-T_xf}_2\lesssim_{s,d}\abs{x}^{\min\{s,1\}}\norm{f}_{H^s(\R^d)}
\end{equation*}
at small spatial scales and optimizing the resulting scale-dependent estimate for the residual energy. The saturation at $s=1$ reflects that the proof uses a first-order translation modulus; no additional universal gain is asserted for regularities above one.

For particular filter geometries, substantially faster decay is known. Waldspurger related high-order wavelet scattering energy to the high-frequency tail of the input and obtained exponential estimates for suitable one-dimensional wavelet systems and Sobolev inputs \cite{Waldspurger2017}. Exponential decay for time-frequency scattering associated with uniform covering frames was established by Czaja and Li \cite{CzajaLi2019}, while Wiatowski, Grohs, and Bölcskei derived polynomial and, for certain wavelet filter banks, exponential bounds under additional analyticity and high-pass assumptions \cite{wiatowski2017energy}. More recently, Führ and the present author developed a unifying framework for these earlier results that also encompasses other important classes of filter banks, including novel wavelet filter banks, as well as higher-dimensional extensions. Within this framework, fast, potentially exponential, energy decay can be guaranteed for signals belonging to a generalized Sobolev class adapted to the underlying filter bank \cite{FuhrGetter2025}. These results naturally raise the question whether exponential energy decay might follow automatically once the input is sufficiently regular.

Our third result answers this question negatively in the strongest natural regularity class considered here.

\begin{theorem}\label{thm:intro-no-universal-exponential-decay}
    There exist a scattering filter bank $(\chi,\Psi)$ of band-limited Schwartz functions on $\R$ and a band-limited Schwartz function $f\in\calS(\R)$ such that
    \[\wh\chi(\xi)=1\]
    on a neighborhood of the origin, while
    \begin{equation*}
        \limsup_{N\to\infty}\norm{U_Nf}_{\ell^2(\PathDepth{N};L^2(\R))}^{1/N}=1.
    \end{equation*}
\end{theorem}

Thus, even for a band-limited Schwartz input and an ideal low-pass response near the origin, the propagated energy need not satisfy any asymptotic estimate of the form $\calO(\alpha^N)$ with $\alpha<1$. Corollary~\ref{cor:intro-Hs-polynomial-decay} therefore cannot be replaced, at this level of generality, by a universal exponential bound. This obstruction is complementary to the arbitrarily slow energy propagation established for wavelet scattering on dense subsets of $L^2(\R^d)$ in \cite{FuhrGetter2025}: there the filter bank is constrained by a prescribed wavelet geometry while the constructed examples have very weak regularity, whereas Theorem~\ref{thm:intro-no-universal-exponential-decay} constructs the filter bank so that subexponential behavior already occurs for one fixed band-limited Schwartz function.

The counterexample combines classical Fourier techniques with the scattering architecture. Specially designed channels isolate long, nearly flat blocks of Fourier coefficients of $\abs{\cos(2\pi\cdot)}$. After phase alignment, these blocks produce functions with the profile of the modulus of a Dirichlet kernel, whose energy is concentrated in a prescribed intermediate frequency range. Repeated application of the corresponding single-step propagator preserves a fixed proportion of the energy over arbitrarily many layers, yielding decay slower than any exponential rate.

\subsection{Outline}
The remainder of this paper is organized as follows. Section~\ref{sec:universal-admissibility} proves Theorem~\ref{thm:intro-universal-admissibility}, establishing the admissibility of arbitrary scattering filter banks. Section~\ref{sec:a-universal-Sobolev-decay-estimate} derives the generic polynomial decay bound stated in Corollary~\ref{cor:intro-Hs-polynomial-decay}. Finally, Section~\ref{sec:failure-of-universal-exponential-decay} constructs the counterexample to exponential energy decay, thereby establishing Theorem~\ref{thm:intro-no-universal-exponential-decay}.

\section{Universal admissibility}\label{sec:universal-admissibility}

The proof of Theorem~\ref{thm:intro-universal-admissibility} rests on the following correlation inequality, which is the only step in which the nonvanishing assumption $\wh\chi(0)\neq0$ enters. Throughout, $T_xh\coloneqq h(\cdot-x)$ denotes translation by $x\in\R^d$.

\begin{lemma}\label{lem:correlation}
    Let $(\chi,\Psi)$ be a scattering filter bank on $\R^d$. Then there exist a radius $r_\chi>0$, constants $\kappa,C_\chi>0$, and a function $w\in L^1(\R^d)\cap C_0(\R^d)$ such that $w\geq0$ on $\R^d$,
    \[w(x)\geq\kappa\qquad\text{for every }x\in B_{r_\chi}(0),\]
    and
    \[\int_{\R^d}w(x)\ip{h}{T_xh}\dx\leq C_\chi\norm{h*\chi}_2^2\]
    for every nonnegative $h\in L^2(\R^d)$.
\end{lemma}

\begin{proof}
    Recall from \eqref{eq:low-pass_pointwise_lower_Fourier_bound} that there exist a radius $R_\chi>0$ and a constant $c_\chi>0$ such that
    \begin{equation*}
        \abs{\wh\chi(\xi)}\geq c_\chi \qquad \text{for every } \xi\in B_{R_\chi}(0).
    \end{equation*}
    Let $r_0\in(0,R_\chi/2)$. Choose a nonzero function $a\in C_c(\R^d)$ satisfying $a\geq0$ and $\supp a\subset B_{r_0}(0)$, and define $w$ in the Fourier domain by
    \begin{equation*}
        \widetilde a=\overline{a(-\,\cdot)}=a(-\,\cdot),\qquad \wh w=a*\widetilde a.
    \end{equation*}
    Then $\wh w\in C_c(\R^d)$ is nonnegative and even, with
    \[\supp\wh w\subset\overline{B_{2r_0}(0)}\subset B_{R_\chi}(0).\]
    Moreover,
    \[w=\check a\check{\widetilde a}=\abs{\check a}^2\geq0.\]
    Since $a\in L^2(\R^d)$, Plancherel's theorem gives $\check a\in L^2(\R^d)$, and hence $w=\abs{\check a}^2\in L^1(\R^d)$. Moreover, $\wh w\in C_c(\R^d)\subset L^1(\R^d)\cap L^\infty(\R^d)$, so Fourier inversion yields $w\in C_0(\R^d)$. Finally,
    \[w(0)=\abs{\int_{\R^d}a(\xi)\dxi}^2>0.\]
    By continuity of $w$, there exist $r_\chi,\kappa>0$ such that
    \[w(x)\geq\kappa \qquad \text{for every } x\in B_{r_\chi}(0).\]
    Since $\supp\wh w\subset B_{R_\chi}(0)$, the pointwise lower bound \eqref{eq:low-pass_pointwise_lower_Fourier_bound} implies that
    \[\wh w(\xi)\leq C_\chi\abs{\wh\chi(\xi)}^2\]
    for every $\xi\in\R^d$, where one may take $C_\chi=c_\chi^{-2}\norm{\wh w}_{L^\infty}$.
    
    Now, let $h\in L^2(\R^d)$ be nonnegative. Then $\ip{h}{T_xh}\geq0$ for every $x\in\R^d$, and Plancherel's theorem gives
    \begin{equation*}
        \ip{h}{T_xh}=\int_{\R^d}e^{2\pi\iu x\cdot\xi}\abs{\wh h(\xi)}^2\dxi.
    \end{equation*}
    Since $w\in L^1(\R^d)$ and $\abs{\wh h}^2\in L^1(\R^d)$, the corresponding double integral is absolutely integrable. Fubini's theorem, the evenness of $\wh w$, and the preceding pointwise estimate therefore yield
    \[\int_{\R^d}w(x)\ip{h}{T_xh}\dx=\int_{\R^d}\wh w(\xi)\abs{\wh h(\xi)}^2\dxi\leq C_\chi\int_{\R^d}\abs{\wh\chi(\xi)}^2\abs{\wh h(\xi)}^2\dxi=C_\chi\norm{h*\chi}_2^2.\]
\end{proof}

With the correlation estimate established, we now turn to the proof of universal admissibility.

\begin{proof}[Proof of Theorem~\ref{thm:intro-universal-admissibility}]
    Fix $f\in L^2(\R^d)$. We must show that 
    \[W_\infty(f)=\lim_{N\to\infty}W_N(f)=0,\]
    where
    \[W_N(f)=\sum_{p\in\PathDepth{N}}\norm{U[p]f}_2^2.\]
    For $x\in\R^d$ and $N\in\Nzero$, set
    \[D_N(x)\coloneqq\sum_{p\in\PathDepth{N}}\norm{T_xU[p]f-U[p]f}_2^2.\]
    Translations commute with convolution, and the modulus acts pointwise and is nonexpansive. Therefore,
    \begin{equation*}
        D_{N+1}(x)\leq\sum_{p\in\PathDepth{N}}\sum_{\lambda\in\Lambda}\norm{(T_xU[p]f-U[p]f)*\psi_\lambda}_2^2\leq D_N(x),
    \end{equation*}
    where the second inequality follows from the Parseval identity \eqref{eq:LP-condition}. Consequently,
    \begin{equation}\label{eq:D_NvsD_0}
        D_N(x)\leq D_0(x)=\norm{T_xf-f}_2^2.
    \end{equation}
    For $N\geq1$, every $U[p]f$ with $p\in\PathDepth{N}$ is nonnegative. Set
    \[I_N(x)\coloneqq\sum_{p\in\PathDepth{N}}\ip{U[p]f}{T_xU[p]f}.\]
    For every $x\in\R^d$, nonnegativity and Cauchy--Schwarz give
    \[0\leq I_N(x)\leq\sum_{p\in\PathDepth{N}}\norm{U[p]f}_2\norm{T_xU[p]f}_2=W_N(f)<\infty.\]
    Expanding the square and using the translation invariance of the $L^2$-norm therefore yields
    \[D_N(x)=2W_N(f)-2I_N(x).\]
    Let $r_\chi$, $\kappa$, $C_\chi$, and $w$ be as in Lemma~\ref{lem:correlation}.
    Applying Lemma~\ref{lem:correlation} to $h=U[p]f$, summing over $p\in\PathDepth{N}$, and using Tonelli's theorem for the nonnegative integrands, we obtain
    \begin{equation}\label{eq:aux-int_vanishes_as_N_to_infty}
        \begin{aligned}
            \int_{\R^d}w(x)I_N(x)\dx&=\sum_{p\in\PathDepth{N}}\int_{\R^d}w(x)\ip{U[p]f}{T_xU[p]f}\dx\\&\leq C_\chi\sum_{p\in\PathDepth{N}}\norm{(U[p]f)*\chi}_2^2 \\
            &=C_\chi A_N(f)\xrightarrow{N\to\infty}0.
        \end{aligned}
    \end{equation}
    It remains to convert the decay of the averaged correlations in \eqref{eq:aux-int_vanishes_as_N_to_infty} into the vanishing of the residual energy.

    Fix $\eps>0$. By the strong continuity of translations on $L^2(\R^d)$, there exists $\delta>0$ such that
    \[\norm{T_xf-f}_2^2<\eps\qquad\text{for every }x\in B_\delta(0).\]
    Choose $r>0$ such that
    \[B_r(0)\subset B_{r_\chi}(0)\cap B_\delta(0).\]
    Lemma~\ref{lem:correlation} and \eqref{eq:aux-int_vanishes_as_N_to_infty} imply that
    \begin{equation*}
        \int_{B_r(0)}I_N(x)\dx\leq\frac{1}{\kappa}\int_{\R^d}w(x)I_N(x)\dx\xrightarrow{N\to\infty}0.
    \end{equation*}
    Integrating the identity $2W_N(f)=D_N(x)+2I_N(x)$ over $B_r(0)$ and using~\eqref{eq:D_NvsD_0} gives
    \begin{equation*}
        \begin{aligned}
            2\abs{B_r(0)}W_N(f)&\leq\int_{B_r(0)}\norm{T_xf-f}_2^2\dx+2\int_{B_r(0)}I_N(x)\dx\\&\leq\eps\abs{B_r(0)}+2\int_{B_r(0)}I_N(x)\dx.
        \end{aligned}
    \end{equation*}
    Letting $N\to\infty$ in the preceding estimate yields
    \[2W_\infty(f)\leq \eps.\]
    Since $\eps>0$ is arbitrary, $W_\infty(f)=0$, and hence
    \[\lim_{N\to\infty}\norm{U_Nf}_{\ell^2(\PathDepth{N};L^2(\R^d))}=0.\]
    Finally, letting $N\to\infty$ in \eqref{eq:finite-energy-telescoping} yields
    \[\norm{\Scat f}_{\ell^2(\PathsAll;L^2(\R^d))}^2=\sum_{n=0}^\infty\sum_{p\in\PathDepth{n}}\norm{\Scat[p]f}_2^2=\norm{f}_2^2.\]
    Thus $\Scat$ is norm-preserving.
\end{proof}

\begin{remark}
    The proof of Theorem~\ref{thm:intro-universal-admissibility} is not intrinsically Euclidean: the qualitative argument extends to arbitrary locally compact abelian groups. More precisely, let $\G$ be a locally compact abelian group, equip $\G$ and its Pontryagin dual $\Gdual$ with dual Haar measures for which Plancherel's theorem holds, and suppose that $\chi,\psi_\lambda\in L^1(\G)\cap L^2(\G)$ satisfy
    \[\abs{\wh\chi(\xi)}^2+\sum_{\lambda\in\Lambda}\abs{\wh\psi_\lambda(\xi)}^2=1\qquad\text{for almost every }\xi\in\Gdual,\]
    with $\wh\chi(0)\neq0$. Then the scattering filter bank defined using convolution on $\G$ is admissible on $L^2(\G)$. Indeed, continuity of $\wh\chi$ provides an identity neighborhood in $\Gdual$ on which $\abs{\wh\chi}$ is bounded from below. Choosing a nonnegative function $a\in C_c(\Gdual)$ whose difference support lies in this neighborhood and setting $w\coloneqq \abs{\check a}^2$ gives the analogue of Lemma~\ref{lem:correlation}. Continuity of $w$ and $w(0)>0$ provide a relatively compact identity neighborhood in $\G$ on which $w$ is bounded from below. The remainder of the proof uses only Plancherel's theorem, translation covariance, strong continuity of translations on $L^2(\G)$, and Haar invariance. Nevertheless, we restrict our attention to $\R^d$, as the quantitative estimates below rely on its Euclidean structure.
\end{remark}

\section{A universal Sobolev decay estimate}\label{sec:a-universal-Sobolev-decay-estimate}

We now return to the Euclidean setting and exploit its quantitative translation geometry to derive a universal decay rate for the residual norm of Sobolev inputs, thereby proving Corollary~\ref{cor:intro-Hs-polynomial-decay}.

\begin{proof}[Proof of Corollary~\ref{cor:intro-Hs-polynomial-decay}]
    For $s\geq1$, the assertion follows directly from the case $s=1$ and the continuous embedding $H^s(\R^d)\hookrightarrow H^1(\R^d)$. It therefore suffices to consider $0<s\leq1$, for which we prove that the residual energy satisfies $W_N(f)\lesssim \norm{f}_{H^s(\R^d)}^2 N^{-2s/d}$.

    Throughout the proof, the symbols $\lesssim$ and $\gtrsim$ denote inequalities up to positive multiplicative constants depending only on $s$, $d$, and $\chi$. Fix $f\in H^s(\R^d)$. 

    For $N\in \N$ and $x\in\R^d$, we define again
    \begin{equation*}
        D_N(x)\coloneqq\sum_{p\in\PathDepth{N}}\norm{T_xU[p]f-U[p]f}_2^2,\qquad I_N(x)\coloneqq\sum_{p\in\PathDepth{N}}\ip{U[p]f}{T_xU[p]f}.
    \end{equation*}
    Recall from the proof of Theorem~\ref{thm:intro-universal-admissibility} that
    \[D_N(x)\leq\norm{T_xf-f}_2^2,\qquad D_N(x)=2W_N(f)-2I_N(x).\]
    Let $r_\chi$, $\kappa$, $C_\chi$, and $w$ be given by Lemma~\ref{lem:correlation}. For every $0<r\leq r_\chi$, Tonelli's theorem and Lemma~\ref{lem:correlation} give
    \begin{equation*}
        \begin{aligned}
            \int_{B_r(0)}I_N(x)\dx
            &=\sum_{p\in\PathDepth{N}}
            \int_{B_r(0)}\ip{U[p]f}{T_xU[p]f}\dx\\
            &\leq\frac{1}{\kappa}
            \sum_{p\in\PathDepth{N}}
            \int_{B_r(0)}w(x)\ip{U[p]f}{T_xU[p]f}\dx\\
            &\leq\frac{C_\chi}{\kappa}
            \sum_{p\in\PathDepth{N}}\norm{(U[p]f)*\chi}_2^2\\
            &=\frac{C_\chi}{\kappa}A_N(f).
            \end{aligned}
    \end{equation*}
    Let $\omega_d\coloneqq\abs{B_1(0)}$. Integrating the identity $2W_N(f)=D_N(x)+2I_N(x)$ over $B_r(0)$ yields
    \begin{equation*}
        2\omega_dr^dW_N(f)\leq\int_{B_r(0)}\norm{T_xf-f}_2^2\dx+\frac{2C_\chi}{\kappa}A_N(f).
    \end{equation*}
    By Plancherel's theorem and the estimate $\abs{e^{-2\pi\iu t}-1}\lesssim\abs{t}^s$, we have
    \begin{equation*}
        \norm{T_xf-f}_2^2=\int_{\R^d}\abs{e^{-2\pi\iu x\cdot\xi}-1}^2\abs{\wh f(\xi)}^2\dxi\lesssim\abs{x}^{2s}\norm{f}_{H^s(\R^d)}^2.
    \end{equation*}
    Consequently,
    \begin{equation*}
        \int_{B_r(0)}\norm{T_xf-f}_2^2\dx\lesssim r^{d+2s}\norm{f}_{H^s(\R^d)}^2,
    \end{equation*}
    and hence, for every $0<r\leq r_\chi$,
    \begin{equation}\label{eq:Hs-scale-estimate}
        W_N(f)\lesssim r^{2s}\norm{f}_{H^s(\R^d)}^2+r^{-d}A_N(f).
    \end{equation}

    We next optimize \eqref{eq:Hs-scale-estimate} over the admissible spatial scales. If $A_N(f)=0$, then letting $r\downarrow0$ in \eqref{eq:Hs-scale-estimate} gives $W_N(f)=0$. Suppose therefore that $A_N(f)>0$. Then $\norm{f}_{H^s(\R^d)}>0$, and we set
    \[\rho_N\coloneqq\left(\frac{A_N(f)}{\norm{f}_{H^s(\R^d)}^2}\right)^{\frac{1}{d+2s}}.\]
    We claim that
    \begin{equation}\label{eq:Hs-dissipation-lower-bound}
    A_N(f)\gtrsim\norm{f}_{H^s(\R^d)}^{-\frac{d}{s}}W_N(f)^{1+\frac{d}{2s}}.
    \end{equation}
    
    Suppose first that $\rho_N\leq r_\chi$. Choosing $r=\rho_N$ in \eqref{eq:Hs-scale-estimate} balances its two terms and gives
    \[W_N(f)\lesssim\norm{f}_{H^s(\R^d)}^{\frac{2d}{d+2s}}A_N(f)^{\frac{2s}{d+2s}}.\]
    Raising this inequality to the power $\frac{d+2s}{2s}$ and rearranging yields \eqref{eq:Hs-dissipation-lower-bound}.
    
    Suppose instead that $\rho_N>r_\chi$. Then
    \[A_N(f)>\norm{f}_{H^s(\R^d)}^2r_\chi^{d+2s},\]
    and therefore
    \[r_\chi^{2s}\norm{f}_{H^s(\R^d)}^2<r_\chi^{-d}A_N(f).\]
    Choosing $r=r_\chi$ in \eqref{eq:Hs-scale-estimate} consequently gives
    \[W_N(f)\lesssim r_\chi^{-d}A_N(f)\lesssim A_N(f).\]
    On the other hand, $W_N(f)\leq W_0(f)=\norm{f}_2^2\leq\norm{f}_{H^s(\R^d)}^2$, and hence
    \[\norm{f}_{H^s(\R^d)}^{-\frac{d}{s}}W_N(f)^{1+\frac{d}{2s}}=W_N(f)\left(\frac{W_N(f)}{\norm{f}_{H^s(\R^d)}^2}\right)^{\frac{d}{2s}}\leq W_N(f)\lesssim A_N(f).\]
    This proves \eqref{eq:Hs-dissipation-lower-bound} in both cases.

    We may assume that $N\geq 2$ and $W_N(f)>0$, since otherwise the desired estimate is immediate. The monotonicity of $(W_n(f))_{n\in\Nzero}$ implies that $W_n(f)>0$ for every $1\leq n\leq N$. For each $1\leq n\leq N-1$, the lower bound established in \eqref{eq:Hs-dissipation-lower-bound} gives
    \begin{equation*}
        \begin{aligned}
            W_{n+1}(f)^{-\frac{d}{2s}}-W_n(f)^{-\frac{d}{2s}}&=\frac{d}{2s}\int_{W_{n+1}(f)}^{W_n(f)}t^{-1-\frac{d}{2s}}\dd t\\&\geq\frac{d}{2s}(W_n(f)-W_{n+1}(f))W_n(f)^{-1-\frac{d}{2s}}\\&\gtrsim \norm{f}_{H^s(\R^d)}^{-\frac{d}{s}}.
        \end{aligned}
    \end{equation*}
    Summing these inequalities and telescoping yields
    \[W_N(f)^{-\frac{d}{2s}}-W_1(f)^{-\frac{d}{2s}}=\sum_{n=1}^{N-1}\left(W_{n+1}(f)^{-\frac{d}{2s}}-W_n(f)^{-\frac{d}{2s}}\right)\gtrsim(N-1)\norm{f}_{H^s(\R^d)}^{-\frac{d}{s}}.\]
    In particular,
    \[W_N(f)\lesssim \norm{f}_{H^s(\R^d)}^2(N-1)^{-\frac{2s}{d}}.\]
    Since $N-1\geq\frac{N}{2}$ for $N\geq2$, taking square roots gives
    \[\norm{U_Nf}_{\ell^2(\PathDepth{N};L^2(\R^d))}=W_N(f)^{\frac{1}{2}}\lesssim \norm{f}_{H^s(\R^d)} N^{-\frac{s}{d}}.\]
    This completes the proof.
\end{proof}

\section{Failure of universal exponential decay}\label{sec:failure-of-universal-exponential-decay}

We now show that, despite the faster decay rates available for structured filter banks, exponential decay cannot hold universally under the assumptions of Corollary~\ref{cor:intro-Hs-polynomial-decay}. To prove Theorem~\ref{thm:intro-no-universal-exponential-decay}, we construct a scattering filter bank of band-limited Schwartz functions and a band-limited Schwartz input whose propagated energy satisfies no exponential upper bound.

\begin{proof}[Proof of Theorem~\ref{thm:intro-no-universal-exponential-decay}]
    The proof proceeds in four steps. First, we show that the modulus of a long Dirichlet kernel has most of its $L^2$-energy in a suitable intermediate frequency block. Second, we choose an input whose first modulus generates infinitely many nearly flat blocks of this type. Third, we construct a scattering filter bank whose distinguished channels isolate, phase-align, and repeatedly retain these blocks. Finally, at each scale we follow a path along which a polynomial amount of energy survives to arbitrarily large depth, thereby ruling out exponential decay.

    \smallskip
    \noindent\emph{Step 1: A Fourier-series estimate.}
    We begin on the torus $\T\coloneqq\R/\Z$, equipped with normalized Haar measure and the standard quotient metric induced by $\norm{\cdot}_\T\coloneqq\dist(\,\cdot,\Z)$.
    For $L\geq2$, let
    \begin{equation*}
        D_L(\theta)=\sum_{m=0}^{L-1}e^{2\pi\iu m\theta},\qquad \abs{D_L(\theta)}=\sum_{n\in\Z}d_{L,n}e^{2\pi\iu n\theta}, \qquad \theta\in \T.
    \end{equation*}
    Using
    \[D_L(\theta)=
    \begin{cases}
        e^{\pi\iu(L-1)\theta}\,\frac{\sin(\pi L\theta)}{\sin(\pi\theta)}, & \theta\neq 0, \\
        L, &\theta=0,
    \end{cases}
    \]
    and the bound
    $\abs{\sin(\pi\theta)}\gtrsim\norm{\theta}_\T$, we obtain
    \[\abs{D_L(\theta)}\lesssim\min\set*{L,\norm{\theta}_\T^{-1}},\qquad\theta\in\T.\]
    Here and in the following, the constants are universal and, in particular, independent of $L$. Integrating the pointwise estimate and using the Fourier coefficient formula yield
    \begin{equation*}
        \abs{d_{L,n}}\leq\norm{\abs{D_L}}_{L^1(\T)}=\norm{D_L}_{L^1(\T)}\lesssim\log(L).
    \end{equation*}
    The function $\abs{D_L}$ is absolutely continuous and satisfies
    \[\abs{(\abs{D_L})^\prime(\theta)}\leq\abs{D_L^\prime(\theta)}\]
    for almost every $\theta\in\T$. Hence Parseval's identity gives
    \begin{equation*}
        \sum_{n\in\Z}n^2\abs{d_{L,n}}^2=\frac{1}{4\pi^2}\norm{(\abs{D_L})^\prime}_{L^2(\T)}^2\leq\frac{1}{4\pi^2}\norm{D_L^\prime}_{L^2(\T)}^2=\sum_{m=0}^{L-1}m^2\lesssim L^3.
    \end{equation*}
    For integers $1\leq A\leq B$, let $P_{A,B}$ denote the Fourier projection on $L^2(\T)$ onto the frequencies in
    \[\set*{n\in\Z\given A\leq\abs{n}\leq B}.\]
    The low-frequency contribution satisfies
    \[\sum_{\abs{n}<A}\abs{d_{L,n}}^2\lesssim A\log^2(L),\]
    whereas the derivative estimate gives
    \begin{equation*}
        \sum_{\abs{n}>B}\abs{d_{L,n}}^2\leq\frac{1}{B^2}\sum_{n\in\Z}n^2\abs{d_{L,n}}^2\lesssim\frac{L^3}{B^2}.
    \end{equation*}
    Since $\norm{\abs{D_L}}_{L^2(\T)}^2=\norm{D_L}_{L^2(\T)}^2=L$, it follows that
    \begin{equation}\label{eq:dirichlet-projection-estimate}
        \frac{\norm{(\Id-P_{A,B})\abs{D_L}}_{L^2(\T)}^2}{\norm{\abs{D_L}}_{L^2(\T)}^2}\lesssim\frac{A\log^2(L)}{L}+\frac{L^2}{B^2}.
    \end{equation}
    We choose the parameters at different integer scales $j\geq 1$ so that both the projection error and the relative variation of the Fourier coefficients are of order $A_j^{-1}$. This will allow us to iterate the resulting channel on the order of $A_j$ times. To this end, let $A_1\geq2$ be an integer whose size will be fixed below, and define recursively, for every integer $j\geq1$,
    \begin{equation}\label{eq:parameter-recursion}
        L_j\coloneqq A_j^{4},\qquad B_j\coloneqq A_jL_j=A_j^{5},\qquad N_j\coloneqq 2B_j,\qquad A_{j+1}\coloneqq N_j+L_j.
    \end{equation}
    Substituting \eqref{eq:parameter-recursion} into \eqref{eq:dirichlet-projection-estimate} and using $\log^2(L_j)\lesssim A_j$, we obtain a universal constant $C_1\geq1$ such that
    \begin{equation}\label{eq:dirichlet-block-projection}
        \frac{\norm{(\Id-P_{A_j,B_j})\abs{D_{L_j}}}_{L^2(\T)}}{\norm{\abs{D_{L_j}}}_{L^2(\T)}}\leq C_1A_j^{-1}
    \end{equation}
    for every $j\geq1$.
    The preceding estimate identifies an intermediate frequency window in which the modulus of a long Fourier block retains almost all of its energy. We next construct an input whose first modulus generates such nearly flat blocks at infinitely many scales.

    \smallskip
    \noindent\emph{Step 2: An input signal.}
    Set $\eps=1/4$. Choose a nonzero function $a\in C_c^\infty((-\eps/2,\eps/2))$ and set
    \[\phi=\abs{\check a}^2.\]
    Then $\phi\in\calS(\R)$, $\phi\geq0$, and
    \[\supp\wh\phi\subset[-\eps,\eps].\]
    Define
    \begin{equation}\label{eq:input-function}
        f(x)=\phi(x)\cos(2\pi x),\quad x\in\R.
    \end{equation}
    Thus $f$ is a nonzero band-limited Schwartz function.

    The Fourier series of $\abs{\cos(2\pi\cdot)}$ is
    \begin{equation}\label{eq:absolute-cosine-series}
        \abs{\cos(2\pi \cdot)}=\sum_{n\in\Z}c_ne^{4\pi\iu n\cdot},\qquad c_n=\frac{2(-1)^{n+1}}{\pi(4n^2-1)}.
    \end{equation}
    In particular,
    \begin{equation}\label{eq:cosine-coefficient-decay}
        c_n\neq0,\qquad \abs{c_n}\asymp(1+\abs{n})^{-2}.
    \end{equation}
    For $j\geq1$, define
    \begin{equation}\label{eq:Gj-definition}
        G_j\coloneqq\abs{\sum_{m=0}^{L_j-1}\frac{\abs{c_{N_j+m}}}{\abs{c_{N_j}}}e^{2\pi\iu m\,\cdot}}.
    \end{equation}
    For $0\leq m<L_j$, the explicit formula in \eqref{eq:absolute-cosine-series} gives
    \[\frac{\abs{c_{N_j+m}}}{\abs{c_{N_j}}}=\frac{4N_j^2-1}{4(N_j+m)^2-1}.\]
    Since $L_j<N_j$, it follows that
    \[\abs{\frac{\abs{c_{N_j+m}}}{\abs{c_{N_j}}}-1}=\frac{8N_jm+4m^2}{4(N_j+m)^2-1}\lesssim\frac{m}{N_j}\leq\frac{L_j}{N_j}.\]
    Using \eqref{eq:parameter-recursion}, we therefore obtain a universal constant $C_2\geq1$ such that
    \begin{equation}\label{eq:coefficient-flatness}
    \max_{0\leq m<L_j}\abs{\frac{\abs{c_{N_j+m}}}{\abs{c_{N_j}}}-1}\leq C_2\frac{L_j}{N_j}\leq\frac{C_2}{A_j}.
    \end{equation}
    Since the modulus is nonexpansive, Parseval's identity and \eqref{eq:coefficient-flatness} yield
    \begin{equation}\label{eq:Gj-dirichlet-approximation}
        \norm{G_j-\abs{D_{L_j}}}_{L^2(\T)}\leq\norm{\sum_{m=0}^{L_j-1}\left(\frac{\abs{c_{N_j+m}}}{\abs{c_{N_j}}}-1\right)e^{2\pi\iu m\,\cdot}}_{L^2(\T)}\leq C_2A_j^{-1}L_j^{1/2}.
    \end{equation}
    Set $C_3=2(C_1+C_2)$ and, once and for all, choose the integer $A_1\geq2$ sufficiently large that $A_1\geq4C_3$, hence $C_2A_1^{-1}\leq 1/2$.
    Since the sequence $(A_j)_{j\in\N}$ is increasing, these inequalities remain valid with $A_1$ replaced by any $A_j$. It follows from \eqref{eq:Gj-dirichlet-approximation} and $\norm{\abs{D_{L_j}}}_{L^2(\T)}=L_j^{1/2}$ that
    \begin{equation*}
        \norm{G_j}_{L^2(\T)}\geq\left(1-C_2A_j^{-1}\right)L_j^{1/2}\geq\frac12L_j^{1/2}.
    \end{equation*}
    Moreover, by \eqref{eq:dirichlet-block-projection} and \eqref{eq:Gj-dirichlet-approximation},
    \begin{equation*}
        \norm{(\Id-P_{A_j,B_j})G_j}_{L^2(\T)}\leq(C_1+C_2)A_j^{-1}L_j^{1/2}.
    \end{equation*}
    Dividing this estimate by the preceding lower bound proves
    \begin{equation}\label{eq:Gj-projection-estimate}
        \frac{\norm{(\Id-P_{A_j,B_j})G_j}_{L^2(\T)}}{\norm{G_j}_{L^2(\T)}}\leq C_3A_j^{-1},\qquad \norm{G_j}_{L^2(\T)}\geq\frac12L_j^{1/2}.
    \end{equation}
    The estimate \eqref{eq:Gj-projection-estimate} therefore provides the frequency blocks that the scattering channels must preserve. We now construct a scattering filter bank containing distinguished filters adapted to these blocks.

    \smallskip
    \noindent\emph{Step 3: Construction of a filter bank.}
    Choose an even function $\vartheta\in C^\infty(\R;[0,\pi/2])$ such that
    \begin{equation*}
        \vartheta(\xi)=0\quad\text{for }\abs{\xi}\leq \eps,\qquad \vartheta(\xi)=\pi/2\quad\text{for }\abs{\xi}\geq 2\eps,
    \end{equation*}
    and set
    \begin{equation}\label{eq:low-pass-definition}
        \wh\chi \coloneqq\cos\vartheta,\qquad \eta\coloneqq\sin\vartheta.
    \end{equation}
    Then $\wh\chi\in C_c^\infty(\R)$, $\wh\chi=1$ on $[-\eps,\eps]$, and
    \begin{equation}\label{eq:low-high-partition}
        \abs{\wh\chi(\xi)}^2+\eta(\xi)^2=1.
    \end{equation}

    For $n\in\Z$, define
    \[I_n\coloneqq [2n-\eps,2n+\eps],\]
    and let
    \begin{equation*}
        E_j\coloneqq \set*{n\in\Z\given N_j\leq n<N_j+L_j},\qquad R_j\coloneqq \set*{n\in\Z\given A_j\leq\abs{n}\leq B_j}.
    \end{equation*}
    Consider the compact sets
    \begin{equation}\label{eq:distinguished-compact-sets}
            K_\ast\coloneqq \left[-1-\eps,-1+\eps\right]\cup \left[1-\eps,1+\eps\right], \quad
            K_{e_j}\coloneqq \bigcup_{n\in E_j}I_n, \quad\text{and} \quad
            K_{r_j}\coloneqq \bigcup_{n\in R_j}I_n.
    \end{equation}
    The recursion \eqref{eq:parameter-recursion} gives
    \[B_j<N_j,\qquad N_j+L_j-1<A_{j+1},\]
    and hence the sets in \eqref{eq:distinguished-compact-sets}, taken over all $j\geq1$, are pairwise disjoint. Figure~\ref{fig:counterexample-frequency-schematic} illustrates these sets and their role in the construction below.  

    We now construct a countable family $(q_\lambda)_{\lambda\in\Lambda}\subset C_c^\infty(\R)$ containing distinguished functions $q_\ast,q_{e_j},q_{r_j}$, for all $j\geq 1$, such that
    \begin{equation}\label{eq:q-quadratic-partition}
        \sum_{\lambda\in\Lambda}\abs{q_\lambda(\xi)}^2=1\qquad\text{for every }\xi\in\R,
    \end{equation}
    and
    \begin{equation}\label{eq:distinguished-q-values}
        \begin{aligned}
            q_\ast&=1&&\text{on }K_\ast,\\
            q_{e_j}&=\frac{\overline{c_n}}{\abs{c_n}}&&\text{on }I_n,\quad n\in E_j,\\
            q_{r_j}&=1&&\text{on }K_{r_j}.
        \end{aligned}
    \end{equation}
    In addition, $q_{e_j}$ vanishes on $I_n$ for $n\notin E_j$, and $q_{r_j}$ vanishes on $I_n$ for $n\notin R_j$.


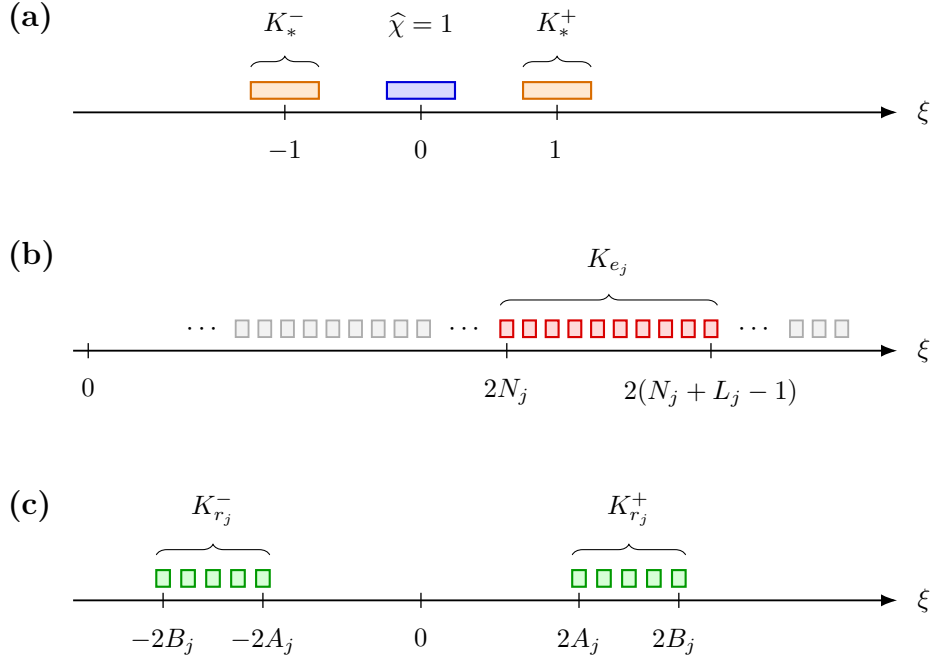
\begin{figure}[t]
\centering
\begin{tikzpicture}[
    x=1cm,
    y=1cm,
    >={Latex},
    font=\small,
    axis/.style={line width=0.7pt,->},
    tick/.style={line width=0.5pt},
    brace/.style={
        decorate,
        decoration={brace,amplitude=4pt}
    },
    orangebox/.style={
        draw=orange!85!black,
        fill=orange!18,
        line width=0.8pt
    },
    bluebox/.style={
        draw=blue!85!black,
        fill=blue!15,
        line width=0.8pt
    },
    redbox/.style={
        draw=red!85!black,
        fill=red!15,
        line width=0.8pt
    },
    greenbox/.style={
        draw=green!60!black,
        fill=green!16,
        line width=0.8pt
    },
    graybox/.style={
        draw=gray!65,
        fill=gray!10,
        line width=0.7pt
    },
    insetbox/.style={
        draw=black!55,
        rounded corners=2pt,
        line width=0.55pt
    },
    trapped/.style={
        draw=gray!50,
        fill=gray!9,
        rounded corners=2pt,
        line width=0.65pt
    },
    panel heading/.style={
        anchor=west,
        font=\bfseries\large,
        align=left
    },
    description/.style={
        anchor=west,
        align=left
    },
    every node/.style={align=center}
]


\def\PanelLeft{0.10}
\def\AxisStart{1.10}
\def\AxisEnd{12.00}
\def\AxisLabel{12.13}

\coordinate (panel-a-origin) at (0,0);
\coordinate[below=3.15cm of panel-a-origin] (panel-b-origin);
\coordinate[below=3.30cm of panel-b-origin] (panel-c-origin);


\begin{scope}[shift={(panel-a-origin)}]

    \node[panel heading] at (\PanelLeft,1.25) {(a)};

    \draw[axis] (\AxisStart,0) -- (\AxisEnd,0);
    \node[anchor=west] at (\AxisLabel,0) {$\xi$};

    \def\InputMinus{3.90}
    \def\InputZero{5.70}
    \def\InputPlus{7.50}

    \foreach \x/\lbl in {
        \InputMinus/$-1$,
        \InputZero/$0$,
        \InputPlus/$\:1$
    }{
        \draw[tick] (\x,-0.10) -- (\x,0.10);
        \node[below=4pt] at (\x,-0.10) {\lbl};
    }

    \draw[orangebox]
        ({\InputMinus-0.45},0.18) rectangle ++(0.90,0.22);
    \draw[bluebox]
        ({\InputZero-0.45},0.18) rectangle ++(0.90,0.22);
    \draw[orangebox]
        ({\InputPlus-0.45},0.18) rectangle ++(0.90,0.22);

    \pgfmathsetmacro{\InputMinusLeft}{\InputMinus-0.45}
    \pgfmathsetmacro{\InputMinusRight}{\InputMinus+0.45}
    \draw[brace]
        (\InputMinusLeft,0.6) -- (\InputMinusRight,0.6);
    \node[above=3pt]
        at (\InputMinus,0.75) {$K_\ast^-$};

    \node[above=3pt]
        at (\InputZero,0.75) {$\wh\chi=1$};

    \pgfmathsetmacro{\InputPlusLeft}{\InputPlus-0.45}
    \pgfmathsetmacro{\InputPlusRight}{\InputPlus+0.45}
    \draw[brace]
        (\InputPlusLeft,0.6) -- (\InputPlusRight,0.6);
    \node[above=3pt]
        at (\InputPlus,0.75) {$K_\ast^+$};

\end{scope}


\begin{scope}[shift={(panel-b-origin)}]

    \node[panel heading]
        at (\PanelLeft,1.25) {(b)};

    \draw[axis] (\AxisStart,0) -- (\AxisEnd,0);
    \node[anchor=west] at (\AxisLabel,0) {$\xi$};

    \def\TickZeroB{1.3}
    \draw[tick] (\TickZeroB,-0.10) -- (\TickZeroB,0.10);
    \node[below=4pt] at (\TickZeroB,-0.10) {$0$};


    \def\PacketWidth{0.17}
    \def\PacketStep{0.30}
    \def\FirstGrayStart{3.25}
    \def\RedStart{6.75}
    \def\LastGrayStart{10.58}

    \node at (2.83,0.29) {$\cdots$};

    \foreach \k in {0,...,8}{
        \draw[graybox]
            ({\FirstGrayStart+\PacketStep*\k},0.18)
            rectangle ++(\PacketWidth,0.22);
    }

    \node at (6.3,0.29) {$\cdots$};

    \foreach \k in {0,...,9}{
        \draw[redbox]
            ({\RedStart+\PacketStep*\k},0.18)
            rectangle ++(\PacketWidth,0.22);
    }

    \foreach \k in {0,1,2}{
        \draw[graybox]
            ({\LastGrayStart+\PacketStep*\k},0.18)
            rectangle ++(\PacketWidth,0.22);
    }

    \pgfmathsetmacro{\RedEnd}{
        \RedStart+9*\PacketStep+\PacketWidth
    }
    \pgfmathsetmacro{\RedCenter}{0.5*(\RedStart+\RedEnd)}

    \draw[brace]
        (\RedStart,0.6) -- (\RedEnd,0.6);
    \node[above=3pt]
        at (\RedCenter,0.75) {$K_{e_j}$};

    \draw[tick] (\RedStart+\PacketWidth/2,-0.10) -- (\RedStart+\PacketWidth/2,0.10);
    \draw[tick] (\RedEnd-\PacketWidth/2,-0.10) -- (\RedEnd-\PacketWidth/2,0.10);

    \node[below=4pt]
        at (\RedStart+\PacketWidth/2,-0.10) {$2N_j$};
    \node[below=4pt]
        at (\RedEnd-\PacketWidth/2,-0.10) {$2(N_j+L_j-1)$};

    \node at (10.13,0.29) {$\cdots$};

\end{scope}


\begin{scope}[shift={(panel-c-origin)}]

    \node[panel heading]
        at (\PanelLeft,1.25) {(c)};

    \def\TickMinusB{2.29}
    \def\TickMinusA{3.61}
    \def\TickZero{5.70}
    \def\TickPlusA{7.79}
    \def\TickPlusB{9.11}

    \draw[axis] (\AxisStart,0) -- (\AxisEnd,0);
    \node[anchor=west] at (\AxisLabel,0) {$\xi$};

    \foreach \x/\lbl in {
        \TickMinusB/$-2B_j$,
        \TickMinusA/$-2A_j$,
        \TickZero/$0$,
        \TickPlusA/$2A_j$,
        \TickPlusB/$2B_j$
    }{
        \draw[tick] (\x,-0.10) -- (\x,0.10);
        \node[below=4pt] at (\x,-0.10) {\lbl};
    }

    \def\GreenWidth{0.18}
    \def\GreenStep{0.33}
    \def\NegativeGreenStart{2.20}
    \def\PositiveGreenStart{7.70}

    \foreach \k in {0,...,4}{
        \draw[greenbox]
            ({\NegativeGreenStart+\GreenStep*\k},0.18)
            rectangle ++(\GreenWidth,0.22);
    }

    \foreach \k in {0,...,4}{
        \draw[greenbox]
            ({\PositiveGreenStart+\GreenStep*\k},0.18)
            rectangle ++(\GreenWidth,0.22);
    }

    \pgfmathsetmacro{\NegativeGreenEnd}{
        \NegativeGreenStart+4*\GreenStep+\GreenWidth
    }
    \pgfmathsetmacro{\PositiveGreenEnd}{
        \PositiveGreenStart+4*\GreenStep+\GreenWidth
    }
    \pgfmathsetmacro{\NegativeGreenCenter}{
        0.5*(\NegativeGreenStart+\NegativeGreenEnd)
    }
    \pgfmathsetmacro{\PositiveGreenCenter}{
        0.5*(\PositiveGreenStart+\PositiveGreenEnd)
    }

    \draw[brace]
        (\NegativeGreenStart,0.6)
        -- (\NegativeGreenEnd,0.6);
    \node[above=3pt]
        at (\NegativeGreenCenter,0.75) {$K_{r_j}^{-}$};

    \draw[brace]
        (\PositiveGreenStart,0.6)
        -- (\PositiveGreenEnd,0.6);
    \node[above=3pt]
        at (\PositiveGreenCenter,0.75) {$K_{r_j}^{+}$};

\end{scope}
\end{tikzpicture}

\caption{Frequency-space schematic of the slowly decaying scattering paths; the panels are not drawn to a common scale. (a) The input satisfies $\supp(\wh{f})\subset K_\ast$, where $\wh\psi_\ast|_{K_\ast}=1$ acts as the identity; the blue block records that $\wh{\chi}=1$ near the origin. (b) The output $U[\psi_\ast]f$ consists of packets $I_n=[2n-\eps,2n+\eps]$, of which $\psi_{e_j}$ selects and phase-aligns $K_{e_j}$. (c) The resulting function $g_j=U[\psi_\ast,\psi_{e_j}]f$ has relative $L^2$-leakage of order $A_j^{-1}$ outside $K_{r_j}$. Since $\smash{\wh\psi_{r_j}=1}$ on $K_{r_j}$, the channel $U[\psi_{r_j}]$ can be iterated $m_j\asymp A_j$ times while retaining a fixed proportion of $\norm{g_j}_2$. The TikZ implementation was developed with assistance from ChatGPT (GPT-5.6 Sol) and subsequently revised by the author.}
\label{fig:counterexample-frequency-schematic}
\end{figure}

    Denote by $\calJ$ the collection of all intervals that occur as connected components of the sets in \eqref{eq:distinguished-compact-sets}. For each interval $J\in \calJ$, set
    \[O(J)\coloneqq \set*{\xi\in\R\given\dist(\xi,J)<\eps}.\]
    The components have centers separated by at least $2$. Each component has radius $\eps$, and each neighborhood $O(J)$ has radius $2\eps$. Since $4\eps<2$, the neighborhoods $O(J)$ are pairwise disjoint. Moreover, the parameter recursion implies that their centers tend to infinity in absolute value, so every bounded subset of $\R$ meets only finitely many of them.

    For each interval $J\in\calJ$, choose $\theta_J\in C_c^\infty(\R;[0,1])$ with $\supp\theta_J\Subset O(J)$ and with $\theta_J$ identically one in a neighborhood of $J$. Denote the two components of $K_\ast$ by $J_+$ and $J_-$. Define
    \begin{equation*}
        \begin{aligned}
            q_\ast&\coloneqq\sum_{\sigma\in\set*{+,-}}\sin\left(\frac{\pi}{2}\theta_{J_\sigma}\right), \\
            q_{r_j}&\coloneqq\sum_{n\in R_j}\sin\left(\frac{\pi}{2}\theta_{I_n}\right),\\
            q_{e_j}&\coloneqq\sum_{n\in E_j}\frac{\overline{c_n}}{\abs{c_n}}\sin\left(\frac{\pi}{2}\theta_{I_n}\right).
        \end{aligned}
    \end{equation*}
    Since the neighborhoods $O(J)$ are pairwise disjoint and each displayed sum is finite, these are smooth compactly supported functions satisfying \eqref{eq:distinguished-q-values} and the stated vanishing properties.
    
    Define $b\in C^\infty(\R;[0,1])$ by
    \begin{equation*}
        b(\xi)\coloneqq
        \begin{cases}
            \displaystyle
            \cos\left(\frac{\pi}{2}\theta_J(\xi)\right),
            & \xi\in O(J)\text{ for some }J\in\calJ,\\[1ex]
            1,
            & \displaystyle
            \xi\notin\bigcup_{J\in\calJ}O(J).
        \end{cases}
    \end{equation*}
    The pairwise disjointness of the neighborhoods and the fact that each $\theta_J$ is compactly supported in $O(J)$ show that $b$ is indeed smooth. Let
    \[\Lambda_0\coloneqq\set*{\ast}\cup\set*{e_j,r_j\given j\in\N}.\]
    At every $\xi\in O(J)$, precisely one function $q_\lambda$, $\lambda\in\Lambda_0$, is active, and hence
    \[\abs{b(\xi)}^2+\sum_{\lambda\in\Lambda_0}\abs{q_\lambda(\xi)}^2=\sin^2\left(\frac{\pi}{2}\theta_J(\xi)\right)+\cos^2\left(\frac{\pi}{2}\theta_J(\xi)\right)=1.\]
    Outside the union of the neighborhoods $O(J)$, all distinguished functions vanish and $b=1$, so the same identity holds there.

    It remains only to decompose $b$ into compactly supported terms. Let
    $(\rho_k)_{k\in\Z}\subset C_c^\infty(\R)$ be a locally finite family such that
    \[\sum_{k\in\Z}\abs{\rho_k(\xi)}^2=1\qquad\text{for every }\xi\in\R,\]
    and define $q_k\coloneqq b\rho_k$. Then
    \[\sum_{k\in\Z}\abs{q_k(\xi)}^2=\abs{b(\xi)}^2.\]
    After adjoining these functions to the distinguished family, \eqref{eq:q-quadratic-partition} holds for the countable index set
    \[\Lambda\coloneqq\Z\sqcup\Lambda_0.\]
    For every $\lambda\in\Lambda$, define $\psi_\lambda$ in the Fourier domain by
    \begin{equation}\label{eq:filter-multiplier-definition}
        \wh{\psi_\lambda}\coloneqq \eta \cdot q_\lambda.
    \end{equation}
    Each multiplier in \eqref{eq:filter-multiplier-definition} belongs to $C_c^\infty(\R)$, as does $\wh\chi$. Hence 
    \[\chi,\psi_\lambda\in\calS(\R)\subset L^1(\R)\cap L^2(\R).\] 
    By \eqref{eq:low-high-partition} and \eqref{eq:q-quadratic-partition},
    \begin{equation*}
        \abs{\wh\chi(\xi)}^2+\sum_{\lambda\in\Lambda}\abs{\wh{\psi_\lambda}(\xi)}^2=\abs{\wh\chi(\xi)}^2+\eta(\xi)^2\sum_{\lambda\in\Lambda}\abs{q_\lambda(\xi)}^2=1
    \end{equation*}
    for every $\xi\in\R$. Thus $(\chi,\Psi)$, with $\Psi=(\psi_\lambda)_{\lambda\in\Lambda}$, is a scattering filter bank. With the filter bank constructed, it remains to exhibit scattering paths along which a nonexponentially small amount of energy survives to arbitrarily large depths.

    \smallskip
    \noindent\emph{Step 4: A slowly decaying path.}

    Since
    \[\wh f(\xi)=\frac12\wh\phi(\xi-1)+\frac12\wh\phi(\xi+1),\]
    the Fourier support of $f$ is contained in $K_\ast$. Moreover, $\eta q_\ast=1$ on $K_\ast$, and therefore $f*\psi_\ast=f$.

    By \eqref{eq:cosine-coefficient-decay}, the series in \eqref{eq:absolute-cosine-series} converges absolutely and uniformly. Since $\phi\in L^2(\R)$, multiplication by $\phi$ consequently gives the following convergence in $L^2(\R)$:
    \begin{equation}\label{eq:first-modulus-expansion}
    U[\psi_\ast]f=\phi\abs{\cos(2\pi\,\cdot)}=\sum_{n\in\Z}c_n\phi(\,\cdot\,)e^{4\pi\iu n\,\cdot}.
    \end{equation}
    The Fourier support of the $n$th summand is contained in $I_n$. 
    
    By the definition of $\psi_{e_j}$, for any $j\in\N$,
    \begin{equation}\label{eq:gj-path-output}
        \begin{aligned}
            U[\psi_\ast,\psi_{e_j}]f&=\abs{\sum_{n=N_j}^{N_j+L_j-1}\abs{c_n}\phi(\cdot)\, e^{4\pi\iu n\,\cdot}}\\
            &=\abs{c_{N_j}}\phi(\cdot) G_j(2 \,\cdot)\\&=:g_j.
        \end{aligned}
    \end{equation}

    Since $\supp\wh\phi\subset[-\eps,\eps]$ and $2\eps<2$, the supports of
    $T_{2n}\wh\phi$ and $T_{2m}\wh\phi$ are disjoint whenever
    $n,m\in\Z$ and $n\neq m$. Hence
    \[\ip{T_{2n}\wh\phi}{T_{2m}\wh\phi}=0.\]

    Let $H\in L^2(\T)$ have Fourier expansion
    \[H=\sum_{n\in\Z}h_ne^{2\pi\iu n\,\cdot}.\]
    This series converges unconditionally in $L^2(\T)$. Since the Fourier supports of the functions $e^{4\pi\iu n\,\cdot}\phi$, $n\in\Z$, are pairwise disjoint, these functions form an orthogonal family in $L^2(\R)$, with
    \[\norm{e^{4\pi\iu n\,\cdot}\phi}_2=\norm{\phi}_2.\]
    Consequently, the series
    \[\sum_{n\in\Z}h_ne^{4\pi\iu n\,\cdot}\phi\]
    converges unconditionally in $L^2(\R)$. Its limit agrees with $\phi(\cdot)H(2\,\cdot)$. Indeed, an $L^2(\T)$-convergent sequence of Fourier partial sums has a subsequence converging almost everywhere on $\T$, and the corresponding dilated products converge almost everywhere on $\R$; uniqueness of the $L^2(\R)$ limit gives the identification. Parseval's identity therefore yields
    \begin{equation}\label{eq:dilated-orthogonality}
        \norm{\phi(\cdot)H(2\,\cdot)}_2^2=\sum_{n\in\Z}\abs{h_n}^2\norm{\phi}_2^2=\norm{\phi}_2^2\norm{H}_{L^2(\T)}^2.
    \end{equation}
    In particular, 
    \begin{equation}\label{eq:gj-norm}
        \norm{g_j}_2=\abs{c_{N_j}}\norm{\phi}_2\norm{G_j}_{L^2(\T)}.
    \end{equation}
    Next, observe that the Fourier support of $\phi(\cdot) G_j(2\,\cdot)$ is contained in $\bigcup_{n\in\Z} I_n$. The multiplier $\wh{\psi_{r_j}}$ equals $1$ on $I_n$ whenever $A_j\leq\abs{n}\leq B_j$ and vanishes on every interval $I_n$ with $\abs{n}<A_j$ or $\abs{n}>B_j$. Consequently, 
    \begin{equation}\label{eq:rj-projection-identity}
        \bigl(\phi(\cdot) G_j(2\,\cdot)\bigr)*\psi_{r_j}=\phi(\cdot)\bigl(P_{A_j,B_j}G_j\bigr)(2\,\cdot).
    \end{equation}
    Using \eqref{eq:Gj-projection-estimate}, \eqref{eq:dilated-orthogonality}, and \eqref{eq:gj-norm}, we obtain
    \begin{equation}\label{eq:gj-rj-approximation}
        \begin{aligned}
            \norm{g_j*\psi_{r_j}-g_j}_2&=\abs{c_{N_j}}\norm{\phi(\cdot) (\Id-P_{A_j,B_j})G_j(2\,\cdot)}_2 \\
            &=\abs{c_{N_j}}\norm{\phi}_2\norm{(\Id-P_{A_j,B_j})G_j}_{L^2(\T)} \\
            &=\frac{\norm{(\Id-P_{A_j,B_j})G_j}_{L^2(\T)}}{\norm{G_j}_{L^2(\T)}}\norm{g_j}_2\\
            &\leq C_3A_j^{-1}\norm{g_j}_2.
        \end{aligned}
    \end{equation}
    Since $g_j\geq0$, the nonexpansiveness of the modulus gives
    \begin{equation*}
        \norm{U[\psi_{r_j}]g_j-g_j}_2=\norm{\abs{g_j*\psi_{r_j}}-\abs{g_j}}_2\leq\norm{g_j*\psi_{r_j}-g_j}_2.
    \end{equation*}
    Together with \eqref{eq:gj-rj-approximation}, this yields
    \begin{equation}\label{eq:UvsId}
        \norm{U[\psi_{r_j}]g_j-g_j}_2\leq C_3A_j^{-1}\norm{g_j}_2.
    \end{equation}
    For $m\in\N$, define
    \begin{equation*}
        (U[\psi_{r_j}])^{(m)}\coloneqq\underbrace{U[\psi_{r_j}]\cdots U[\psi_{r_j}]}_{m\text{ times}}, \quad (U[\psi_{r_j}])^{(0)}\coloneqq\Id_{L^2(\R)}.
    \end{equation*}
    By the nonexpansiveness of the single-step scattering propagator $U[\psi_{r_j}]$, each increment in the following telescoping sum is bounded by $\|U[\psi_{r_j}]g_j-g_j\|_2$. Hence \eqref{eq:UvsId} gives
    \begin{equation}\label{eq:Tj-iteration-estimate}
        \norm{(U[\psi_{r_j}])^{(m)}g_j-g_j}_2\leq \sum_{k=1}^m \norm{(U[\psi_{r_j}])^{(k)}g_j-(U[\psi_{r_j}])^{(k-1)}g_j}_2 \leq mC_3A_j^{-1}\norm{g_j}_2.
    \end{equation}
    For every $j\in \N$, set
    \begin{equation}\label{eq:mj-definition}
        m_j=\left\lfloor\frac{A_j}{2C_3}\right\rfloor\asymp A_j.
    \end{equation}
    Our choice $A_1\geq4C_3$ ensures that $m_j\geq1$. 

    It remains to estimate the size of $g_j$. Plugging the relations 
    \[|c_{N_j}|\gtrsim N_j^{-2}\gtrsim A_j^{-10} \quad \text{and} \quad \norm{G_j}_{L^2(\T)}\gtrsim L_j^{1/2}=A_j^2\] 
    into \eqref{eq:gj-norm} shows that there exists a constant $c_0>0$, independent of $j$, such that
    \begin{equation}\label{eq:gj-lower-bound}
        \norm{g_j}_2\geq 2 c_0A_j^{-8}.
    \end{equation}
    Combining this with \eqref{eq:Tj-iteration-estimate} and \eqref{eq:mj-definition}, we obtain
    \begin{equation}\label{eq:Tj-lower-bound}
        \norm{(U[\psi_{r_j}])^{(m_j)}g_j}_2\geq\norm{g_j}_2-\norm{(U[\psi_{r_j}])^{(m_j)}g_j-g_j}_2\geq\frac12\norm{g_j}_2.
    \end{equation}
    The path
    \begin{equation*}
        p_j=\bigl(\psi_\ast,\psi_{e_j},\underbrace{\psi_{r_j},\ldots,\psi_{r_j}}_{m_j\text{ times}}\bigr)
    \end{equation*}
    has length $m_j+2$. By \eqref{eq:gj-lower-bound} and \eqref{eq:Tj-lower-bound},
    \[\norm{U_{m_j+2}f}_{\ell^2(\PathDepth{m_j+2};L^2(\R))}\geq\norm{U[p_j]f}_2=\norm{(U[\psi_{r_j}])^{(m_j)}g_j}_2\geq c_0 A_j^{-8}.\]
    Since $m_j\asymp A_j\to\infty$, the previous bound yields
    \begin{equation*}
        \limsup_{N\to\infty}\norm{U_Nf}_{\ell^2(\PathDepth{N};L^2(\R))}^{1/N}\geq\lim_{j\to\infty}\left(c_0 A_j^{-8}\right)^{\frac{1}{m_j+2}}=1.
    \end{equation*}
    On the other hand, the Parseval property of the filter bank gives
    \[\norm{U_Nf}_{\ell^2(\PathDepth{N};L^2(\R))}\leq\norm{f}_2\]
    for every $N\geq1$. Taking $N$th roots shows that the corresponding limsup is at most $1$, which proves the claim.
\end{proof}

\section*{Use of generative AI}
The author used ChatGPT (GPT-5.6 Sol, OpenAI), accessed through \url{https://chatgpt.com} during July and August 2026, in iterative dialogues concerning proof strategies and mathematical exposition. In particular, after the author had formulated and conjectured the main result in Theorem~\ref{thm:intro-universal-admissibility}, the tool suggested the underlying weighted-correlation proof mechanism. It also contributed a proof strategy for a preliminary version of Corollary~\ref{cor:intro-Hs-polynomial-decay}, and assisted with the iterated-channel argument controlling the propagation of energy through deeper network layers in Section~\ref{sec:failure-of-universal-exponential-decay} and the TikZ implementation of Figure~\ref{fig:counterexample-frequency-schematic}. The prompts included arguments and constructions previously developed or proposed by the author. 

The LLM-generated arguments served as starting points rather than final proofs. The author verified every mathematical statement and argument influenced by LLM-generated material, supplied missing computations and intermediate steps, replaced several proposed arguments with more elementary proofs, and substantially rewrote and reorganized the resulting exposition. The author takes full responsibility for the accuracy, integrity, and originality of the final manuscript, including all mathematical claims, proofs, formulations, and conclusions.

\section*{Acknowledgments}
This work was funded by the Deutsche Forschungsgemeinschaft (DFG, German Research Foundation)---Project-ID 442047500---SFB 1481, Sparsity and Singular Structures. The author thanks Hartmut F\"{u}hr for discussions concerning cosine-modulated constructions that informed the counterexample in Section~\ref{sec:failure-of-universal-exponential-decay}, and Ay\c{s}eg\"{u}l Alabal{\i}k for carefully reading an earlier draft of this paper and providing helpful comments.

\bibliographystyle{plain}
\bibliography{bib}
\vspace{1em}
\end{document}